\documentclass{amsart}

\usepackage{amsmath}
\usepackage{amsthm}
\usepackage{hyperref}
\usepackage{amsfonts,graphics,amsthm,amsfonts,amscd,latexsym}
\usepackage{epsfig}
\usepackage{flafter}
\usepackage{mathtools}
\usepackage{comment}
\usepackage{stmaryrd}

\usepackage{mathabx,epsfig}

\hypersetup{
    colorlinks=true,    
    linkcolor=blue,          
    citecolor=blue,      
    filecolor=blue,      
    urlcolor=blue           
}
\usepackage{tikz}
\usetikzlibrary{graphs,positioning,arrows,shapes.misc,decorations.pathmorphing}

\tikzset{
    >=stealth,
    every picture/.style={thick},
    graphs/every graph/.style={empty nodes},
}

\tikzstyle{vertex}=[
    draw,
    circle,
    fill=black,
    inner sep=1pt,
    minimum width=5pt,
]
\usepackage[position=top]{subfig}
\usepackage{amssymb}
\usepackage{color}

\calclayout

\usetikzlibrary{decorations.pathmorphing}
\tikzstyle{printersafe}=[decoration={snake,amplitude=0pt}]

\newcommand{\supp}{\operatorname{supp}}

\newcommand{\pp}{\mathbb{P}}

\newcommand{\qq}{\mathbb{Q}}
\newcommand{\zz}{\mathbb{Z}}

\newcommand{\kk}{\mathbb{K}}

\def\O#1.{\mathcal {O}_{#1}}			
\def\pr #1.{\mathbb P^{#1}}				
\def\af #1.{\mathbb A^{#1}}			
\def\ses#1.#2.#3.{0\to #1\to #2\to #3 \to 0}	
\def\xrar#1.{\xrightarrow{#1}}			
\def\K#1.{K_{#1}}						
\def\bA#1.{\mathbf{A}_{#1}}			
\def\bM#1.{\mathbf{M}_{#1}}				
\def\bL#1.{\mathbf{L}_{#1}}				
\def\bB#1.{\mathbf{B}_{#1}}				
\def\bK#1.{\mathbf{K}_{#1}}			
\def\subs#1.{_{#1}}					
\def\sups#1.{^{#1}}

\usepackage{tikz}
\usetikzlibrary{matrix,arrows,decorations.pathmorphing}

  \newtheorem{theorem}{Theorem}[section]
  \newtheorem{lemma}[theorem]{Lemma}

  \newtheorem{definition}[theorem]{Definition}
  \newtheorem{example}[theorem]{Example}

  \newtheorem{question}[theorem]{Question}

\newtheorem{remark}[theorem]{Remark}

\theoremstyle{remark}

\numberwithin{equation}{section}

\usepackage[all]{xy}

\begin{document}

\title[Birational complexity of log Calabi--Yau $3$-folds]{Birational complexity of log Calabi--Yau $3$-folds}

\author[J.~Moraga]{Joaqu\'in Moraga}
\address{UCLA Mathematics Department, Box 951555, Los Angeles, CA 90095-1555, USA
}
\email{jmoraga@math.ucla.edu}

\subjclass[2020]{Primary 14E05, 14J32;
Secondary 14E30.}

\begin{abstract}
We study the birational complexity
of log Calabi--Yau $3$-folds.
For such a pair $(X,B)$ 
of index one 
and coregularity zero,
we show that 
$c_{\rm bir}(X,B)\in \{0,2,3\}$.
Further, we prove that $(X,B)$
has a log Calabi--Yau 
crepant birational model 
that admits a crepant 
contraction to $(\pp^{3-c},H_0+\dots+H_{3-c})$,
where $c=c_{\rm bir}(X,B)$.
To prove this, we give a geometric characterization of standard $\pp^1$-links.
\end{abstract}

\maketitle

\setcounter{tocdepth}{1} 
\tableofcontents

\section{Introduction}
For a general type variety $X$
its canonical model is 
the best possible birational model
from almost any perspective.
On the contrast, log Calabi--Yau pairs $(X,B)$ 
may have countably many birational non-isomorphic models 
that satisfy different geometric properties.
Thus, it is natural to introduce and study invariants 
that allows us to detect better birational models
of a given log Calabi--Yau pair.
Toric log Calabi--Yau pairs, due to their 
combinatorial nature, are the log Calabi--Yau pairs that we can understand the most.
The complexity $c(X,B)$, of a log Calabi--Yau pair $(X,B)$, is defined to be the value $\dim X +\dim_\qq {\rm Cl}_\qq(X) -|B|$, 
where $|B|$ is the sum of the coefficients of $B$.
This invariant is non-negative and $(X,\lfloor B\rfloor)$ is toric whenever it is zero (see~\cite{BMSZ18}).
In~\cite{MM24}, Mauri and the author introduced the birational complexity $c_{\rm bir}(X,B)$ of a log Calabi--Yau pair and proved several fundamental properties of this invariant. 
The birational complexity is the minimum among the complexities of birational models of a given log Calabi--Yau pair. This invariant 
detects whether a log Calabi--Yau pair has a {\em toric model} i.e., it is birationally equivalent to a toric log Calabi--Yau pair.
Thus, the birational complexity gives a natural invariant 
to find optimal models for log Calabi--Yau pairs.

In this article, we study the possible values of the birational complexity of log Calabi--Yau $3$-folds $(X,B)$ of index one
and coregularity zero.~\footnote{Log Calabi--Yau pairs of coregularity zero are called pairs of {\em maximal intersection} in the literature.} For each value of the birational complexity, we find a birational model that computes the birational complexity and enjoys some extra geometric features.

Throughout the article, we denote by $\Sigma^n$ the reduced sum of the coordinate hyperplanes of $\mathbb{P}^n$. Gross, Hacking, and Keel proved that every log Calabi--Yau surface of index one and coregularity zero 
admits a toric model, or equivalently, has birational complexity zero (see, e.g.,~\cite[Proposition 1.3]{GHK15}).

\begin{theorem}\label{introthm:logCY-surfaces}
Let $(X,B)$ be a log Calabi--Yau surface of
index one and coregularity zero.
Then, we have $c_{\rm bir}(X,B)=0$.
In particular, there is a crepant birational isomorphism
$(X,B)\simeq_{\rm cbir} (\pp^2,\Sigma^2)$.
\end{theorem} 

Our first theorem is a classification
of the possible values of the birational complexity
of log Calabi--Yau $3$-folds of index one
and coregularity zero.
More precisely, we prove the following.

\begin{theorem}\label{introthm:values-cbir-3-folds}
Let $(X,B)$ be a log Calabi--Yau $3$-fold 
of index one and coregularity zero.
Then, we have $c_{\rm bir}(X,B)\in \{0,2,3\}$.
\end{theorem} 

In Example~\ref{ex:cbir=0}, Example~\ref{ex:cbir=2},
and Example~\ref{ex:cbir=3}, we show that all the values in Theorem~\ref{introthm:values-cbir-3-folds} indeed occur. Further, all such values already appear
for pairs $(X,B)$ where $X$ is a Fano hypersurface in $\pp^4$ and $B$ is a suitable anti-canonical divisor.
The next theorem states that
a log Calabi--Yau $3$-fold 
of index one, coregularity zero, and birational complexity $c$
has a birational model that admits a contraction to $\pp^{3-c}$.

\begin{theorem}\label{introthm:crepant-model-3-fold}
Let $(X,B)$ be a log Calabi--Yau $3$-fold 
of index one and coregularity zero.
Let $c:=c_{\rm bir}(X,B)$ be the birational complexity of the pair.
Then, there is a log Calabi--Yau crepant birational model $(X',B')$ of $(X,B)$ that admits a crepant contraction to $(\pp^{3-c},\Sigma^{3-c})$.
\end{theorem} 

We refer the reader to Definition~\ref{def:crepant-model}
for the concept of crepant birational model
and to Definition~\ref{def:crepant-contraction}
for the concept of crepant contraction.
The previous theorem can be regarded as the fact
that the {\em toric part} of such a log Calabi--Yau $3$-fold can be factored in the base of the fibration
from an appropiate birational model.
To prove the two previous theorems, we will need to study the birational geometry of certain special fibrations; the so-called standard $\pp^1$-links (see Definition~\ref{def:p1-link}).
These fibrations have appeared in the study of dual complexes and log canonical centers~\cite{KK10,FS20}.
Our main theorem in this direction states that a standard $\pp^1$-link is simply the projectivization of a split $\qq$-vector bundle of rank $2$.

\begin{theorem}\label{introthm:p1-link-char}
Let $\phi\colon (X,B)\rightarrow Z$ be a standard $\pp^1$-link.
Then, the variety $X$ is the projectivization over $Z$ of a rank $2$ split $\qq$-vector bundle. Moreover, the divisor $\lfloor B \rfloor$ consists of the reduced sum of the two sections corresponding to the projections to the summands of the split $\qq$-vector bundle.
\end{theorem} 

Using the previous characterization
of standard $\pp^1$-links, we will prove the following theorem about standard $\pp^1$-links of dimension at most three.

\begin{theorem}\label{introthm:bir-model-p1-link}
Let $X$ be a normal variety of dimension at most $3$.
Let $\phi\colon (X,B)\rightarrow Z$ be a standard $\pp^1$-link.
Let $(Z,B_Z)$ be the pair induced on $Z$ by the canonical bundle formula.
Then, we have a commutative diagram:
\[
\xymatrix{
(X,B)\ar[d]_-{\phi} & (Z',B_{Z'}) \times (\pp^1,\Sigma^1)\ar[d]^-{\pi_1}\ar@{-->}[l]_-{\psi} \\
(Z,B_Z) & (Z',B_{Z'})\ar[l]^-{\pi}
}
\]
where $\pi$ is a projective crepant 
birational morphism with $B_{Z'}$ effective, 
$\psi$ is a crepant birational map, 
and $\pi_1$ is the projection to the first component.
\end{theorem}

Projectivizations of split $\qq$-vector bundles over the same base are always birational equivalent to each other.
Thus, we emphasize that the interesting part of Theorem~\ref{introthm:bir-model-p1-link} is that we can realize such birational equivalence by keeping the pair log Calabi--Yau, i.e., not introducing boundary divisor with negative coefficients.

\subsection*{Acknowledgements}

The author would like to thank 
Joshua Enwright, Fernando Figueroa, 
Konstantin Loginov, and Stefano Filipazzi 
for very useful comments.

\section{Preliminaries}
In this section, we recall some preliminary results and prove some lemmas. For the concepts of singularities of the MMP, we refer the reader to~\cite{Kol13}.
For the concepts related to the canonical bundle formula, we refer the reader to~\cite{Amb04,FG12,Fil18}.
We work over an algebraically closed field $\kk$ of characteristic zero.
Given a pair $(X,B)$ and a closed irreducible subvariety $Z$, we write $(X,B;Z)$ for the germ of $(X,B)$ at the generic point of $Z$.

\subsection{Crepant maps}
In this subsection, we recall preliminaries regarding log Calabi--Yau pairs, crepant birational maps, and crepant birational models.

\begin{definition}
{\em 
A {\em contraction} $f\colon X\rightarrow Y$ is a morphism for which $f_*\mathcal{O}_X= \mathcal{O}_Y$.
In particular, if $X$ is normal, then $Y$ is normal.
A {\em fibration} is a contraction with positive dimensional general fiber.
}
\end{definition} 

\begin{definition}
{\em 
Let $f\colon X\rightarrow Y$ be a finite morphism.
Let $(X,B)$ and $(Y,B_Y)$ be two sub-pairs.
We say that $f\colon (X,B)\rightarrow (Y,B_Y)$ is {\em crepant} if $f^*(K_Y+B_Y)=K_X+B$.
Let us emphasize that we ask the equality to hold and not just a $\qq$-linear equivalence.
}
\end{definition} 

\begin{definition}
{\em 
Let $f\colon X\rightarrow Z$ be a projective contraction.
Let $(X,B)$ be a log pair on $X$.
We say that $(X,B)$ is {\em log Calabi--Yau over $Z$} if $(X,B)$ is log canonical and 
$K_X+B\sim_{\qq,Z} 0$.
The {\em index} of $(X,B)$ over $Z$
is the smallest positive integer $i$ for which $i(K_X+B)\sim_Z 0$.
We drop $Z$ from the notation when $Z={\rm Spec}(\kk)$.
}
\end{definition}

\begin{definition}\label{def:crepant-model}
{\em 
Let $X\rightarrow Z$ be a projective contraction
and $\phi\colon X\dashrightarrow X'$
be a birational map over $Z$.
Let $(X,B)$ and $(X',B')$ be two sub-pairs.
We say that $\phi\colon (X,B)\dashrightarrow (X',B')$ is {\em crepant} if the following conditions are satisfied:
\begin{enumerate}
\item there exists a common resolution $Y$ with projective morphisms $p \colon Y\rightarrow X$ and  $q\colon Y\rightarrow X'$, and 
\item we have $p^*(K_X+B)=q^*(K_{X'}+B')$.
\end{enumerate} 
In the previous case, we say that $\phi$ is a {\em crepant birational map} for the sub-pairs $(X,B)$ and $(X',B')$.
Furthermore, we say that $(X,B)$ and $(X',B')$ are {\em crepant birational equivalent} and write 
\[
(X,B)\simeq_{\rm cbir}(X',B').
\]
We say that $(X',B')$ is a {\em crepant birational model} of $(X,B)$ over $Z$.

Given a birational map $X\dashrightarrow X'$ over $Z$ and a sub-pair $(X',B')$, there is a unique sub-pair $(X,B)$ for which $(X,B)\simeq_{\rm cbir}(X',B')$. This sub-pair $(X',B')$ is called the {\em log pull-back} of $(X,B)$ to $X'$.
}
\end{definition}

\begin{remark}{\em 
If $(X,B)$ is log Calabi--Yau over $Z$
and $(X',B')$ is crepant birational equivalent to $(X,B)$ over $Z$, then $(X',B')$ is log Calabi--Yau over $Z$ provided that $B'$ is effective.
Furthermore, the index of $(X,B)$ and $(X',B')$ over $Z$ agree (see, e.g.,~\cite[Lemma 3.2]{FMM22}).
}
\end{remark}

\begin{definition}
{\em 
Let $(X,B)$ be a log Calabi--Yau pair over $Z$.
A crepant birational model $(X',B')$ over $Z$ is said to be a {\em log Calabi--Yau crepant birational model} if $B'\geq 0$, or analogously, if $(X',B')$ is log Calabi--Yau over $Z$.
}
\end{definition}

\begin{definition}
\label{def:crepant-contraction}
{\em 
Let $\phi\colon X\rightarrow Z$ be a projective contraction.
Let $(X,B)$ be a log Calabi--Yau pair over $Z$.
Let $(Z,B_Z)$ be the log canonical pair induced by the canonical bundle formula.
We say that $\phi\colon (X,B)\rightarrow (Z,B_Z)$ is a {\em crepant contraction}.
Analogously, we say that $\phi$ is a {\em crepant contraction} for the log pairs $(X,B)$ and $(Z,B_Z)$.
}
\end{definition}

\begin{definition}
{\em 
Let $(X,B)$ be a log pair.
The {\em coregularity} of $(X,B)$, denoted by ${\rm coreg}(X,B)$, is the dimension of 
a minimal~\footnote{minimal with respect to the inclusion.} log canonical center on a dlt modification of $(X,B)$. 
}
\end{definition}

\subsection{Complexity} In this subsection, we recall the concepts of complexity and birational complexity.

\begin{definition}
{\em
Let $(X,B)$ be a log Calabi--Yau pair.
The {\em complexity} of $(X,B)$ is 
\[
c(X,B):=\dim X + \dim_\qq {\rm Cl}_\qq(X) -|B|,
\]
where $|B|$ is the sum of the coefficients of $B$.
}
\end{definition}

\begin{definition}
{\em 
Let $(X,B)$ be a log Calabi--Yau pair. 
The {\em birational complexity} of $(X,B)$ is
\[
c_{\rm bir}(X,B):=\min \{ c(X',B') \mid 
(X',B')\simeq_{\rm cbir}(X,B)
\text{ and $B'\geq 0$}
\}.
\]
In other words, the birational complexity of $(X,B)$
is the smallest complexity among
log Calabi--Yau crepant birational models
of $(X,B)$.
}
\end{definition} 

By~\cite[Theorem 1.2]{BMSZ18}, we know that $c_{\rm bir}(X,B)\geq 0$ for every log Calabi--Yau pair.

\subsection{Degenerate divisors}

We recall the definition of degenerate divisors for proper surjective morphisms and prove a statement about contraction of degenerate divisors.

\begin{definition}
{\em 
Let $f\colon X\rightarrow Y$ be a proper surjective morphism and $D$ be an effective $\qq$-divisor on $X$. 
\begin{enumerate}
\item[(i)] We say that $D$ is {\em exceptional} over $Y$ if the codimension of $f(D)$ is at least two in $Y$.
\item[(ii)] We say that $D$ is of {\em insufficient fiber type} over a prime divisor $P$ of $Y$ if there is a prime divisor on $X$ that dominates $P$ and is not contained in the support of $D$.
\item[(iii)] We say that $D$ is {\em degenerate} over $Y$ if $D$ is of insufficient fiber type over every prime divisor of $Y$.
\end{enumerate}
}
\end{definition} 

\begin{lemma}\label{lem:contracting-degen}
Let $(X,B)$ be a $\qq$-factorial log canonical pair. 
Let $X\rightarrow Z$ be a contraction of relative dimension at most three.
Assume that $K_X+B\sim_{\qq,Z} E$, where $E$ is an effective divisor that is degenerate over $Z$. 
Then, the $(K_X+B)$-MMP over $Z$ terminates after contracting all the components of $E$.
\end{lemma} 

\begin{proof}
By~\cite[Lemma 2.9]{Lai11}, we know that $\supp(E)\subseteq {\rm Bs}_{-}(K_X+B/Z)$.
Then, we can run a $(K_X+B)$-MMP over $Z$ that contracts all the components of $E$.
Since the relative dimension of $X\rightarrow Z$ is at most three, we know that this MMP terminates with a good minimal model over $Z$. 
Let $X\dashrightarrow X'$ be a birational contraction induced by the MMP on which all the components of $E$ are contracted. If $B'$ is the push-forward of $B$ to $X'$, then $(X',B')$ is log Calabi--Yau over $Z$. Thus, the MMP terminates on this model.
\end{proof}

\begin{lemma}\label{lem:2-to-1}
Let $\phi\colon X\rightarrow Z$ be a Mori fiber space of relative dimension one.
Let $(X,B)$ be a log Calabi--Yau pair over $Z$.
Assume there is a normal prime component $S\subseteq \lfloor B\rfloor$ 
for which $\phi|_S \colon S\rightarrow Z$ is a finite morphism of degree two.
Let $(Z,B_Z)$ be the log Calabi--Yau pair induced by the canonical bundle formula.
Let $Z'\rightarrow Z$ be a birational morphism that only extracts divisors with log discrepancy at most $\frac{1}{2}$ with respect to $(Z,B_Z)$.
Let $(Z',B_{Z'})$ be log pull-back of $(Z,B_Z)$ to $Z'$. 
Then, there exists a crepant birational model $(X',B')$ of $(X,B)$, with $B'$ effective, satisfying the following.
The variety $X'$ admits a Mori fiber space of relative dimension one to $Z''$, where $Z''\rightarrow Z'$ is a dlt modification of $(Z',B_{Z'})$.
\end{lemma} 

\begin{proof}
Let $Y$ be the normalization of the fiber product $S\times_Z X$.
Let $(Y,B_Y)$ be the log pull-back of $(X,B)$ to $(Y,B_Y)$.
Then, the boundary divisor $\lfloor B_Y\rfloor$ contains a prime component that dominates $Y$ with a projective birational morphism.
Let $S'$ be the normalization of the fiber product $Z'\times_Z S$ and $(S',B_{S'})$ be the log pull-back of $(S,B_S)$ to $S'$.
Note that $(S',B_{S'})\rightarrow (Z',B_{Z'})$ is a crepant finite surjective morphism of degree $2$.
Thus, by Riemann-Hurwitz we have that $B_{S'}\geq 0$.
In particular, the morphism
$S'\rightarrow S$ is $\zz_2$-equivariant and only extract log canonical places
and canonical places of $(S,B_S)$.
By adjunction, each such place induces either a log canonical place or a canonical place of $(Y,B_Y)$.
Henceforth, we can find a $\zz_2$-equivariant projective birational morphism $Y'\rightarrow Y$ that extracts all such places.
Thus, we get a $\zz_2$-equivariant fibration $Y'\rightarrow S'$.
Let $(Y',B_{Y'})$ be the log pull-back of $(Y,B_Y)$ to $Y'$.
Without loss of generality, we may assume that $(Y',B_{Y'})$ is log smooth.
The divisor $B_{Y'}$ may have negative coefficients, however, over each prime divisor of $S'$ there is a prime divisor of $B_{Y'}$ with non-negative coefficient.
Let $(X_0,B_0)$ be the quotient of $(Y',B_{Y'})$ by $\zz_2$.
The variety $X_0$ admits a fibration to $Z'$. Some coefficients of $B_0$ may be negative, however, over each prime divisor of $Z'$ there is a prime divisor of $B_0$ with non-negative coefficient.

We construct a divisor $B'_0$ on $X_0$ as follows.
We increase all the negative coefficients of $B_0$ to zero and the coefficients of the prime degenerate divisors of $B_0$ over $Z'$ to one.
By further blowing up, we may assume that $(X_0,B'_0)$ is log smooth.
We run a $(K_{X_0}+B'_0)$-MMP over $Z'$.
This MMP terminates after contracting all the components of $B_0$ with negative coefficients and all the exceptional divisors of $B_0$ over $Z'$ that are not log canonical places. The previous statement follows by Lemma~\ref{lem:contracting-degen}.
Let $X_1\rightarrow Z'$ be the model where this MMP terminates and $B_1$ be the push-forward of $B_0$ to $X_1$.
Then, the pair $(X_1,B_1)$ is log Calabi--Yau over $Z'$ and is crepant birational equivalent to $(X,B)$.
We run a $K_{X_1}$-MMP over $Z'$ that terminates with a Mori fiber space of relative dimension one $X'\rightarrow Z_0$ over $Z'$.
Any divisor extracted by $Z_0\rightarrow Z'$ is the image of a log canonical center of $(X_1,B_1)$. 
Thus, we conclude that $Z_0\rightarrow Z'$ only extract log canonical places of $(Z',B_{Z'})$.
Let $(Z_0,B_{Z_0})$ be the log pull-back of $(Z',B_{Z'})$ to $Z_0$.
Then, the statement follows by applying~\cite[Lemma 2.10]{MM24}
to a dlt modification $Z''$ of $(Z_0,B_{Z_0})$.
\end{proof} 

\subsection{The two-ray game}
In this subsection, we prove two lemmas that are special versions of the two-ray game.

\begin{lemma}\label{lem:going-down}
Let $X$ be a normal variety of dimension at most three.
Let $\phi\colon X\rightarrow Z$ be a Mori fiber space to a $\qq$-factorial variety.
Let $(X,B)$ be a log Calabi--Yau pair over $Z$.
Let $(Z,B_Z)$ be the pair induced on $Z$ by the canonical bundle formula.
Assume that the support of $\phi^*\lfloor B_Z\rfloor$ is contained in the support of $\lfloor B\rfloor$.
Let $Z\rightarrow Z'$ be a divisorial contraction that is defined by an extremal ray $R$ that intersects some prime component of $\lfloor B_Z\rfloor$ positively.
Then, there exists a crepant birational model $(X',B')$ of $(X,B)$, with $B'$ effective, that admits a Mori fiber space to $Z'$.
\end{lemma} 

\begin{proof}
Note that $\rho(X/Z')=2$.
Let $E$ be the divisor being contracted by 
$Z\rightarrow Z'$.
Let $S$ be a prime component of $\lfloor B_Z\rfloor$ that intersects $R$ positively.
Then, for $\epsilon>0$, we have that 
$(X,B-\epsilon \phi^*S)$ is a log canonical pair.
Furthermore, the divisor $K_X+B-\epsilon \phi^* S$ contains $\phi^*E$ on its relative diminished base locus over $Z'$.
Indeed, we have 
$K_X+B-\epsilon \phi^*S \sim_\qq \phi^*(K_Z+B_Z-\delta S)$ for a suitable $\delta>0$ and 
$\phi^*E$ is covered by curves whose images on $Z$ intersect the divisor $K_Z+B_Z-\delta S$ negatively and map to points on $Z'$.
Then, we may run a $(K_X+B-\epsilon \phi^*S)$-MMP over $Z'$ that will terminate with a good minimal model $X'\rightarrow Z'$ after contracting the divisor $\phi^*E$.
This good minimal model has relative Picard rank over $Z'$ and hence it is a Mori fiber space for $K_{X'}$.
\end{proof} 

\begin{lemma}\label{lem:going-down-2}
Let $X$ be a normal variety of dimension at most three.
Let $\phi\colon X\rightarrow Z$ be a Mori fiber space to a $\qq$-factorial variety.
Let $(X,B)$ be a log Calabi--Yau pair over $Z$.
Let $(Z,B_Z)$ be the pair induced on $Z$ by the canonical bundle formula.
Let $Z\rightarrow Z'$ be a divisorial contraction that contracts the curve $C$.
Assume that $C$ contains no log canonical centers of $(Z,B_Z)$.
Then, there is a crepant birational model $(X',B')$ of $(X,B)$, with $B'$ effective, that admits a Mori fiber space to $Z'$.
\end{lemma} 

\begin{proof}
Note that $\rho(X/Z')=2$.
Observe that $(Z,B_Z+\epsilon C)$ is log canonical for $\epsilon>0$ small enough.
Hence, the pair $(X,B+\delta \phi^*C)$ is also log canonical for $\delta>0$ small enough.
The diminished base locus of $K_X+B+\epsilon \delta \phi^*C$ over $Z'$ contains $\phi^*C$.
Indeed, we have 
$K_X+B+\epsilon \delta \phi^*C\sim_\qq \delta \phi^*C$ and the latter is covered by curves that intersect $\phi^*C$ negatively
and map to points on $Z'$.
Hence, we may run a $(K_X+B+\delta \phi^*C)$-MMP over $Z'$
that terminates with a good minimal model
$X'\rightarrow Z'$
after contracting $\phi^*C$.
As $\rho(X'/Z')=1$, this good minimal model is a Mori fiber space for $K_{X'}$.
\end{proof}

\subsection{Torus actions} In this subsection, we prove lemmas regarding torus actions and birational contractions.

\begin{definition}
{\em 
Let $(X,B)$ be a log pair. 
We write ${\rm Aut}(X,B)$ for the {\em  group of automorphisms} of $(X,B)$, i.e., the group
that consists of automorphisms $\phi\colon X\rightarrow X$
for which $\phi^*B=B$.
Let $X\rightarrow Z$ be a contraction.
We write ${\rm Aut}_Z(X,B)$ for the group of automorphisms of $(X,B)$ over $Z$.
}
\end{definition}

\begin{lemma}\label{lem:descending-torus-actions}
Let $\phi\colon X\rightarrow Z$ be a projective morphism. 
Let $\mathbb{T}$ be an algebraic torus.
Let $(X,B)$ be a log canonical pair for which
$\mathbb{T} \leqslant {\rm Aut}_Z(X,B)$.
Let $\phi\colon X\dashrightarrow X'$ be a birational contraction and let $B':=\phi_*B$. 
Assume that $(X',B')$ is plt.
Then, we have $\mathbb{T}\leqslant {\rm Aut}_Z(X',B')$.
\end{lemma}

\begin{proof}
Let $E$ be the reduced exceptional locus of $\phi$.
Note that $\kappa(E)=0$ (see, e.g.,~\cite[Lemma 2.28]{Mor24a}). We argue that $E$ is $\mathbb{T}$-invariant. 
First, note that $t\cdot E \sim E$ for every $t\in \mathbb{T}$ as $\mathbb{T}$ is a connected group. As $\kappa(E)=0$, we get that $\supp(t\cdot E)=\supp(E)$ for every $t\in \mathbb{T}$.
Thus, $t\cdot E=E$ for every $t\in \mathbb{T}$.
Let $Y\rightarrow X$ be a $\mathbb{T}$-equivariant log resolution of $(X,B+\supp(E))$.
Let $\psi\colon Y\rightarrow X'$ be the associated birational contraction.
Note that every exceptional divisor of $\psi$ and the strict transform of $B$ on $Y$ are $\mathbb{T}$-invariant.
Let $E_Y$ be the reduced exceptional divisor of $\psi$ and $B_Y$ be the strict transform of $B$ on $Y$.
Then, we may run a $\mathbb{T}$-equivariant MMP for $K_Y+E_Y+B_Y$ over $Z$. By the plt condition of $(X',B')$ this MMP terminates with a model $Y'\rightarrow Z$ after contracting all the exceptional divisors of $\psi$.
The previous statement follows from Lemma~\ref{lem:contracting-degen}.
Let $B_{Y'}$ be the push-forward of $B_Y$ to $Y'$.
Then, we get $\mathbb{T}\leqslant {\rm Aut}_Z(Y',B_{Y'})$ and $Y'$ is small birational to $X'$ over $Z$.
As $(X',B')$ is plt, the map 
$Y'\dashrightarrow X'$ is a composition of $(K_{Y'}+B_{Y'})$-flops over $Z$.
Each such a flop is defined by a Cartier semiample divisor which is linearly equivalent to a $\mathbb{T}$-invariant Cartier divisor.
Thus, each such a flop is itself $\mathbb{T}$-equivariant. 
We conclude that $\mathbb{T}\leqslant {\rm Aut}_Z(X',B')$.
\end{proof}

\begin{lemma}\label{lem:canonical-places}
Let $X\rightarrow Z$ be a fibration
and $(X,B)$ be a plt pair.
Assume that $\mathbb{G}_m\leqslant {\rm Aut}_Z(X,B)$,
$\lfloor B\rfloor =S_0+S_\infty$, and
$\mathbb{G}_m$ acts as the identity on both
divisors $S_0$ and $S_\infty$.
Then, every exceptional canonical place of $(X,B)$ is $\mathbb{G}_m$-invariant.
\end{lemma}

\begin{proof}
First, we may replace $(X,B)$ with a $\mathbb{G}_m$-equivariant log resolution. By doing so, we replace $(X,B)$ with a sub-plt pair.
Note that $(X,B)$ is sub-klt in the complement
of $S_0+S_\infty$.
Thus, the pair $(X,B)$ has only finitely many canonical valuations on the complement of $S_0+S_\infty$ (see, e.g.,~\cite[Proposition 2.36]{KM92}).
Such canonical valuations can be extracted by recursively blowing-up the intersections of the irreducible components of intersections of divisors in $B^{>0}$.
Such irreducible components are $\mathbb{G}_m$-invariant.
Thus, after finitely many $\mathbb{G}_m$-equivariant blow-ups, we obtain a $\mathbb{G}_m$-equivariant model $(Y,B_Y)$ that is terminal in the complement
of $S_{Y,0}+S_{Y,\infty}$, the strict transform of $S_0+S_\infty$.
Note that $(Y,B_Y)$ is log smooth
and the only exceptional canonical valuations over $Y$
are obtained by blowing-up centers of codimension $2$ contained in either $S_{Y,0}$ or $S_{Y,\infty}$ (see~\cite[Corollary 2.31]{KM92}).
Since $\mathbb{G}_m$ acts as the identity on both $S_{Y,0}$ and $S_{Y,\infty}$,
we conclude that such centers are $\mathbb{G}_m$-invariant. Therefore, every exceptional canonical place of $(Y,B_Y)$ is $\mathbb{G}_m$-invariant.
So, every exceptional canonical place of $(X,B)$ is $\mathbb{G}_m$-invariant.
\end{proof}

\subsection{Standard $\pp^1$-links}
In this subsection, we recall the concepts of $\pp^1$-links and strict conic fibrations, and prove some lemmata regarding these objects.

\begin{definition}\label{def:p1-link}
{\em Let $\phi\colon (X,B)\rightarrow Z$ be a log Calabi--Yau fibration.
We say that $\phi$ is a {\em standard $\pp^1$-link} if the following conditions are satisfied:
\begin{enumerate}
\item the divisor $\lfloor B\rfloor$ consists of two prime divisors $S_0$ and $S_\infty$,
\item the restrictions $\phi|_{S_i}\colon S_i \rightarrow Z$ are isomorphisms for $i\in \{0,\infty\}$,  
\item the pair $(X,B)$ is plt, and
\item each reduced fiber of $\phi$ is isomorphic to $\pp^1$. 
\end{enumerate} 
}
\end{definition} 

\begin{lemma}\label{lem:P1-is-MFS}
Let $\phi\colon (X,B)\rightarrow Z$ be a standard $\pp^1$-link. Assume that $X$ is klt. Then $\phi$ is a Mori fiber space. 
\end{lemma} 

\begin{proof}
Since $X$ is klt, we may run a $K_X$-MMP over $Z$.
Note that every reduced fiber of $\phi$ is isomorphic to $\pp^1$ and intersects both $S_0$ and $S_\infty$.
This follows by conditions (2) and (4) of Definition~\ref{def:p1-link}.
If $X\rightarrow X'$ is either a divisorial contraction or flipping contraction, then the images of $S_0$ and $S_\infty$ would intersect on $X'$, leading to a contradiction. 
We conclude that the $K_X$-MMP over $Z$ terminates with a Mori fiber space $X\rightarrow Z$. Thus $\phi$ is itself a Mori fiber space.
\end{proof}

\begin{lemma}\label{lem:MFS-two-disjoint-sect-is-p1}
Let $\phi\colon X\rightarrow Z$ be a Mori fiber space.
Let $(X,B)$ be a log Calabi--Yau pair over $Z$.
Assume that $(X,B)$ is plt and $\lfloor B\rfloor$ has two disjoint components $S_0$ and $S_\infty$ that are horizontal over $Z$.
Furthermore, assume that $S_0$ and $S_\infty$ are $\qq$-Cartier.
Then, the morphism $\phi\colon (X,B)\rightarrow Z$ is a standard $\pp^1$-link.
\end{lemma} 

\begin{proof}
Note that $S_\infty$ is ample over $Z$.
Hence, all the fibers of $\phi$ are one-dimensional. Otherwise, $S_0$ would intersect $S_\infty$.
In particular, we conclude that $\lfloor B\rfloor$ consists precisely of $S_0$ and $S_\infty$.
Indeed, any other component of $\lfloor B\rfloor$ would be vertical over $Z$ and thus intersect $S_\infty$, violating the plt condition.
This implies condition (1) of Definition~\ref{def:p1-link}.
Since $(X,B)$ is plt, both divisors $S_0$ and $S_\infty$ are normal.
In particular, $(X,B-S_0)$ is plt Fano over $Z$,
so the morphism $\phi|_{S_0}\colon S_0\rightarrow Z$ has connected fibers.
If $\phi|_{S_0}$ contracts a curve, then $S_\infty$ would intersect such curve, contradicting the fact that $S_0$ and $S_\infty$ are disjoint.
Thus, we get that $\phi|_{S_0}\colon S_0\rightarrow S_\infty$ is a bijection between normal varieties, and so it is an isomorphism.
Anlogously, we conclude that $\phi|_{S_\infty}\colon S_\infty\rightarrow Z$ is an isomorphism. This proves condition (2) of Definition~\ref{def:p1-link}.

Finally, it suffices to show that condition (4) of Definition~\ref{def:p1-link} is satisfied.
For this regard, note that $(X,B-\epsilon(S_0+S_\infty))$ is klt and Fano over $Z$.
Thus, by the relative version of Kawamata-Viehweg vanishing, we get $R^1\phi_*\mathcal{O}_X=0$.
In particular, for every fiber $X_z$, we have $H^1(X_z,\mathcal{O}_{X_z})=0$.
Thus, each reduced fiber is isomorphic to a tree of smooth rational curves.
As both $S_0$ and $S_\infty$ are ample over $Z$,
disjoint, and intersect each fiber at a single point, we conclude that each reduced fiber is isomorphic to $\pp^1$.
\end{proof}

\begin{definition}
{\em 
Let $\phi\colon X\rightarrow Z$ be a fibration.
We say that $\phi\colon (X,B)\rightarrow Z$ is a {\em strict conic fibration}
if the following conditions are satisfied:
\begin{enumerate}
\item the morphism $\phi$ is a conic fibration with $\rho(X/Z)=1$,
\item the pair $(X,B)$ is log Calabi--Yau over $Z$,
\item the divisor $\lfloor B\rfloor$ has two prime components $S_0$ and $S_\infty$ that are horizontal over $Z$, and 
\item the prime divisors $S_0$ and $S_\infty$ are disjoint.
\end{enumerate} 
}
\end{definition} 

\begin{lemma}\label{lem:strict-conic-bundle-is-p1-link}
Let $\phi \colon (X,B)\rightarrow Z$ be a strict conic fibration. 
Let $(Z,B_Z)$ be the log pair induced by the canonical bundle formula.
Assume that $(Z,B_Z)$ is $\qq$-factorial and dlt.
Then, there exists $0\leq \Gamma\leq B$ such that 
$\phi \colon (X,\Gamma)\rightarrow Z$ is a standard $\pp^1$-link.
\end{lemma} 

\begin{proof}
Note that for $\epsilon>0$ small enough, the pair $(X,B-\epsilon \phi^*\lfloor B_Z\rfloor)$ is plt.
Set $\Gamma:=B-\epsilon \phi^*\lfloor B_Z\rfloor$.
Note that $\lfloor \Gamma\rfloor$ has two disjoint components that are horizontal over $Z$.
Then, we can apply Lemma~\ref{lem:MFS-two-disjoint-sect-is-p1} to conclude that $\phi\colon (X,\Gamma)\rightarrow Z$ is a standard $\pp^1$-link.
\end{proof} 

\section{\texorpdfstring{Geometry of standard $\pp^1$-links}{Geometry of standard P1-links}}

In this section, we will prove Theorem~\ref{introthm:p1-link-char} that gives a geometric characterization of standard $\pp^1$-links.
First, we prove two lemmata.
The following lemma is proved by Koll\'ar in the last paragraph of the proof of~\cite[Proposition 4.37]{Kol13}.
It states that a $\qq$-factorial standard $\pp^1$-link is locally the finite quotients of a product with $(\pp^1,\Sigma^1)$.

\begin{lemma}\label{lem:local-geometry-P1-link}
Let $X$ be a $\qq$-factorial variety.
Let $\phi\colon (X,B)\rightarrow Z$ be a standard $\pp^1$-link. Let $(Z,B_Z)$ be the log pair induced by the canonical bundle formula.
For any closed point $z\in Z$, there is an affine neighborhood $U$ of $z$ in $Z$
satisfying the following.
Let $(U,B_U)$ be the restriction of $(Z,B_Z)$ to $U$ and $(X_U,B_{X_U})$ be the restriction of 
$(X,B)$ to the preimage $X_U$ of $U$ via $\phi$. 
Then, there is a crepant finite surjective morphism $f\colon (V,B_V)\rightarrow (U,B_U)$
making the following diagram commutative:
\[
\xymatrix{ 
(X_U,B_U)\ar[d]_-{\phi_U} & (V,B_V) \times (\pp^1,\Sigma^1)\ar[d]^-{\pi_1}\ar[l]_-{f_X} \\ 
(U,B_U) & (V,B_V) \ar[l]^-{f} 
}
\]
where $f_X$ is a crepant finite surjective morphism, 
$\phi_U$ is the restriction of $\phi$ to $X_U$, and $\pi_1$ is the projection to the first component.
\end{lemma} 

Conversely, we show that a log Calabi--Yau fibration which is a finite quotient of a product with $(\pp^1,\Sigma^1)$
is indeed a standard $\pp^1$-link.
In the converse statement, we do not require $X$ to be $\qq$-factorial.

\begin{lemma}
\label{lemma:quotient-of-product-is-p1-link}
Let $\phi\colon (X,B)\rightarrow (Z,B_Z)$ be a crepant log Calabi--Yau fibration
with $(Z,B_Z)$ klt.
Assume there exists a commutative diagram of log pairs:
\[
\xymatrix{ 
(X,B)\ar[d]_-{\phi} & (V,B_V)\times (\pp^1,\Sigma^1)\ar[d]^-{\pi_1} \ar[l]_-{f_X} \\
(Z,B_Z) & (V,B_V)\ar[l]^-{f} 
}
\]
where the horizontal arrows are crepant finite surjective morphisms and $\pi_1$ is the projection to the first component.
Then $\phi$ is a standard $\pp^1$-link.
\end{lemma} 

\begin{proof}
Since $(Z,B_Z)$ is klt
and $f$ is crepant, we conclude that
$(V,B_V)$ is klt.
In particular, the pair
$(V,B_V)\times (\pp^1,\Sigma^1)$ is plt.
Since $f_X$ is crepant and finite, we have that $(X,B)$ is plt. 
This gives condition (3) of Definition~\ref{def:p1-link}.
Furthermore, by Riemann-Hurwitz we conclude that $\lfloor B\rfloor$ contains precisely two prime divisors $S_0$ and $S_\infty$, namely the images
of $V\times\{0\}$ and $V\times \{\infty\}$, respectively.
This proves condition (1) of Definition~\ref{def:p1-link}.
By~\cite[Proposition 5.51]{KM92}, we know that $S_0$ and $S_\infty$ are normal varieties
that are disjoint. 
We argue that $S_0$ and $S_\infty$ are $\qq$-Cartier divisors.
Let $\bar{X}\rightarrow X$ be the Galois closure of $f_X$, so we have a commutative diagram:
\[
\xymatrix{
 & (\overline{X},\overline{B})\ar[ld]_-{f_{\overline{X}}} \ar[d]^-{h} \\ 
(X,B) & (V,B_V)\times (\pp^1,\Sigma^1) \ar[l]^-{f_X}
}
\]
of crepant finite surjective morphisms.
In the previous diagram, 
the morphisms
$h$ and $f_{\overline{X}}$ are crepant Galois morphisms and $(\overline{X},\overline{B})$ is a plt sub-pair.
Let $G$ be the Galois group associated with the finite Galois morphism $f_{\overline{X}}$.
Since $V\times \{0\}$ is a Cartier divisors, 
we have that 
$h^*(V\times \{0\})$ is a Cartier divisor on $\overline{X}$.
In particular, $h^*(V\times \{0\})$ is a $G$-invariant Cartier divisor.
Hence, the divisor $S_0$ is $\qq$-Cartier.
Analogously, the divisor $S_\infty$ is $\qq$-Cartier. 

Both $S_0$ and $S_\infty$ are ample over $Z$. Indeed, this follows from the projection formula
and the fact that both $V\times \{0\}$ and 
$V\times \{\infty\}$ are ample over $V$. 
Therefore, the pair $(X,B-S_0)$ is plt and Fano over $Z$.
In particular, the morphism $\phi|_{S_\infty}\colon S_\infty \rightarrow Z$ has connected fibers.
If $\phi|_{S_\infty}\colon S_\infty \rightarrow Z$ contracts a curve, then $S_0$ must intersect such a curve, contradicting the fact that $S_0$ and $S_\infty$ are disjoint.
We conclude that  $\phi|_{S_\infty}\colon S_\infty \rightarrow Z$ is a bijection between normal varieties and so it is an isomorphism.
Analogously, we conclude that $\phi|_{S_0}\colon S_0\rightarrow Z$ is also an isomorphism.
This gives condition (2) of Definition~\ref{def:p1-link}.

Note that all fibers of $\phi$ are curves, otherwise, $S_0$ and $S_\infty$ would intersect as they would induce ample divisors on a positive-dimensional variety.
Applying the relative version of Kawamata-Viehweg vanishing theorem to the klt pair $(X,B-\epsilon_1S_1-\epsilon_2S_2)$,
which is klt Fano over $Z$, we get 
$R^1\phi_*\mathcal{O}_X=0$.
Thus, for every closed point $z\in Z$, the fiber $X_z$ satisfies that 
$H^1(X_z,\mathcal{O}_{X_z})=0$.
Thus, every fiber $X_z$ is a tree of smooth rational curves.
Since both $S_0$ and $S_\infty$ can only intersect each fiber $X_z$ at a single point, 
and they are both ample over $Z$, 
we conclude that each reduced fiber is a rational curve. This gives condition (4) of Definition~\ref{def:p1-link}. 
\end{proof}

Now, we turn to give a proof of the geometric characterization of $\pp^1$-links.

\begin{proof}[Proof of Theorem~\ref{introthm:p1-link-char}]
We set $n:=\dim X$.
We will prove the statement in six steps.
In the first three steps, we perform birational modifications to produce a rank $2$ split vector bundle over a suitable resolution of $Z$.
In the fourth step, we use this vector bundle to endow $X$ with a torus action. In the fifth step, we use the torus action on $X$ to prove the statement in the $\qq$-factorial case. In the last step, we deal with the non-$\qq$-factorial case.\\

\textit{Step 1:} In this step, we produce a $\pp^1$-link on a resolution of singularities of $Z$.\\

Let $S_0$ and $S_\infty$ be the two components of $\lfloor B\rfloor$ that dominate $Z$.
We denote by $(Z,B_Z)$ be the klt pair induced by the canonical bundle formula on $Z$.
Let $Z'\rightarrow Z$ be a log resolution of $(Z,B_Z)$.
Let $X'$ be the normalization of the fiber product $Z'\times_Z X$.
Let $(X',B')$ be the log pull-back of $(X,B)$ to $X'$.
We denote by $S'_0$ and $S'_\infty$ the pull-backs of $S_0$ and $S_\infty$ to $X'$, respectively.
Note that $B'$ may not be an effective divisor, however, all its components with negative coefficients are vertical over $Z'$.
Let $U$ be an open affine in $Z$.
Let $(U,B_U)$ be the restriction of $(Z,B_Z)$ to $U$.
Let $U'$ be the preimage of $U$ in $Z'$
and let $(U',B_{U'})$ be the log pull-back of $(U,B_U)$ to $U'$.
Let $(X_U,B_{X_U})$ be the restriction of 
$(X,B)$ to the preimage of $U$ in $X$.
We denote by $\phi_U\colon X_U\rightarrow U$ the restriction of $\phi$ to $X_U$.
By Lemma~\ref{lem:local-geometry-P1-link}, up to shrinking $U$, we have a commutative diagram:
\[
\xymatrix{ 
(X_U,B_U)\ar[d]_-{\phi_U} & (V,B_V) \times (\pp^1,\Sigma^1)\ar[d]^-{\pi_1}\ar[l]_-{f_X} \\ 
(U,B_U) & (V,B_V) \ar[l]^-{f} 
}
\]
where $f$ and $f_X$ are crepant finite surjective morphism and $\pi_1$ is the projection to the first component.
Let $X'_U$ be the normalization
of $X_U\times_U U'$ 
and let $(X'_U,B_{X'_U})$ be the log pull-back of
$(X_U,B_U)$ to $X'_U$.
Let $V'$ be the normalization of $V\times_U U'$ and let $(V',B_{V'})$ be the log pull-back of $(U,B_U)$ to $V'$.
By the universal property of fiber products, we obtain a commutative diagram of crepant morphisms:
\begin{equation}\label{eq:comm-diag-product}
\xymatrix{
(X_{U}',B_{X_{U}'}) \ar[d]_-{\phi_{U'}} & (V',B_{V'})\times 
(\pp^1,\Sigma^1) \ar[l]_-{g} \ar[d]^-{\pi_1} \\ 
(U',B_{U'}) & (V',B_{V'})\ar[l]^-{h} 
}
\end{equation} 
where $\pi_1$ is the projection to the first component, 
$\phi_{U'}$ is a fibration, and 
$h$ is a crepant finite surjective morphism.

We argue that $g$ is a finite morphism.
Let $S_{X'_U,0}$ and $S_{X'_U,\infty}$ be the 
restrictions of $S'_0$ and $S'_\infty$ to $X'_U$, respectively.
Note that both $S_{X'_U,0}$ and $S_{X'_U,\infty}$ are $\qq$-Cartier.
Furthermore, the divisors
$S_{X'_U,0}$ and $S_{X'_U,\infty}$ pull-back to positive
multiples of $V'\times \{0\}$ and $V'\times \{\infty\}$, respectively. 
Assume that $g$ contracts a curve $C$. 
If $C$ is vertical over $V'$, then $C$ is a fiber of the projection $\pi_1$ and 
\[
0 
= S_{X'_U,0} \cdot g(C) 
= g^*(S_{X'_U,0})\cdot C
= m_0 (V'\times \{0\}) \cdot C 
= m_0 >0,
\]
leading to a contradiction.
If $C$ is horizontal over $V'$, then by the commutativity of the diagram~\eqref{eq:comm-diag-product}, we get that
$h$ contracts a curve. This leads to a contradiction as well. 
We conclude that $g$ is finite.

Note that we have all the conditions to apply Lemma~\ref{lemma:quotient-of-product-is-p1-link}
except possibly one.
Indeed, the couples in the commutative diagram~\eqref{eq:comm-diag-product} may be sub-pairs. Hence, we need to add a suitable effective divisor to make them log pairs.
Note that $(U',B_{U'})$ is a log smooth sub-pair
and the coefficients of $B_{U'}$ are strictly less than one. 
Furthermore, $h$ is a finite morphism. 
We pick $\Gamma\geq 0$ on $U'$, a simple normal crossing divisor, such that
$1-\frac{1}{r_P}\leq {\rm coeff}_P(B_{U'}+\Gamma)<1$
for every prime divisor $P$ of $U'$, where $r_P$ is the ramification index of $h$ at $P$.
Then, we have that both $(U',B'_U+\Gamma)$ and $(V',B_{V'}+h^*\Gamma)$ are a klt pairs.
Furthermore, the pair
$(V',B_{V'}+h^*\Gamma)\times (\pp^1,\Sigma^1)$ is a plt pair, 
so $(X_{U'}, B_{X_{U'}}+\phi_{U'}^*\Gamma)$ is a plt pair. 
Hence, we get a commutative diagram:
\[
\xymatrix{ 
(X'_U, B_{X'_U}+\phi_{U'}^*\Gamma)\ar[d]_-{\phi_{U'}} & (V',B_{V'}+h^*\Gamma)\times (\pp^1,\Sigma^1)\ar[l]_-{g}\ar[d]^-{\pi_1} \\
(U',B_{U'}+\Gamma) & (V',B_{V'}+h^*\Gamma)\ar[l]^-{h}
}
\]
where the horizontal arrows are crepant finite surjective morphisms 
and $\pi_1$ is the projection to the first component. Note that $\phi_{U'}$ is a crepant contraction for the log Calabi--Yau pair
$(X'_U,B_{X'_U}+\phi_{U'}^*\Gamma)$ and $(U',B_{U'}+\Gamma)$ is klt.
By Lemma~\ref{lemma:quotient-of-product-is-p1-link}, we conclude that 
$(X'_{U},B_{X'_{U}}+\phi_{U'}^*\Gamma)\rightarrow U'$ is a standard $\pp^1$-link.
As $U'$ is smooth, all the vertical divisors of $X_{U'}\rightarrow U'$ are $\qq$-Cartier.
This implies that
$(X'_{U},S_{X'_U,0} +S_{X'_U,\infty})\rightarrow U'$ is a standard $\pp^1$-link as well. 
Thus, we conclude that 
$(X',S'_0+S'_\infty)\rightarrow Z'$ is a standard $\pp^1$-link.
This follows, as being a standard $\pp^1$-link is a local property on the base of the fibration.
Note tat both $S'_0$ and $S'_\infty$ are $\qq$-Cartier divisors.\\

\textit{Step 2:} In this step, we modify the standard $\pp^1$-link over $Z'$ such that it has no multiple fibers over codimension one points.\\

Assume that $(X',S_0'+S_\infty')\rightarrow Z'$ has a fiber of multiplicity $m$ over the prime divisor $P\subset Z'$.
Let $mE_P:={\phi'}^*P$.
Then, the variety $X'$ has a toric singularity of orbifold index $m$ at the generic point of $Q:=S_0'\cap {\phi'}^*P$ (see~\cite[Theorem 1]{MS21}).
In particular, there is a blow-up $Y\rightarrow X'$ that extracts a $\qq$-Cartier divisor $E_Q$ with center $Q$ and log discrepancy $\frac{1}{m}$ with respect to the pair $(X',S_0'+S_\infty')$ (see e.g.,~\cite[Theorem 1]{Mor19}).
Let $(Y,S_{Y,0}+S_{Y,\infty}+(1-\frac{1}{m})E_Q)$ be the log pull-back of $(X',S_0'+S_\infty')$ to $Y$.
Note that both $S_{Y,0}$ and $S_{Y,\infty}$ are $\qq$-Cartier divisors.
By the canonical bundle formula, the divisor $E_Q$ appears with multiplicity $1$ in 
${\phi'}^*P$.
Note that the strict transform of $E_P$ in $Y$ is a degenerate divisor over $Z'$.
Hence, we may contract the strict transform 
of $E_P$ over $Z'$ with a birational contraction
$Y\dashrightarrow Y'$.
This birational contraction can be obtained by running a $(K_{Y'}+S'_\infty+(1-\frac{1}{m})E_Q)$-MMP over $Z'$. 
Let $E_{Q,Y'}$ be the strict transform of $E_Q$ in $Y'$.
Let $S_{Y',i}$ be the strict transform of $S'_i$ in $Y'$.
Then, we have a commutative diagram 
\[
\xymatrix{
(X',S_0'+S_\infty')\ar[rd]_-{\phi'} &  & \left(Y',S_{Y',0}+S_{Y',\infty}+\left(1-\frac{1}{m}\right)E_{Q,Y'}\right)\ar@{-->}[ll]_-{\pi}\ar[ld]^-{\phi_{Y'}} \\
& Z' &
}
\]
where $\pi$ is a crepant birational map.
Note that $\phi'$ is a Mori fiber space by Lemma~\ref{lem:P1-is-MFS}. 
Indeed, we know that $S'_0$ and $S'_\infty$ are $\qq$-Cartier.
Henceforth, the morphism $\phi_{Y'}$ is a Mori fiber space. 
Note that $S_{Y',0}$ and $S_{Y',\infty}$ are $\qq$-Cartier.
By Lemma~\ref{lem:MFS-two-disjoint-sect-is-p1}, we conclude that $\phi_{Y'}$ is a standard $\pp^1$-link whose fiber over $P$ is not multiple. Since $\phi'$ has only finitely many multiple fibers over codimension one points, we can repeat the above procedure to obtain a standard $\pp^1$-link $(X'',S_0''+S_\infty'')\rightarrow Z'$ with no multiple fibers over codimension one points.
The pair $(X',S_0'+S'_\infty)$ is crepant equivalent to a log Calabi--Yau pair
$(X'',S_0''+S_\infty''+S)$ over $Z'$, where $S$ is an effective divisor with standard coefficients that is vertical over $Z'$.\\

\textit{Step 3:} In this step, we show that $X''$ is the projectivization of a rank $2$ split vector bundle over $Z'$.\\

First, we argue that $X''$ is smooth.
Let $z\in Z'$ be a closed point. 
Let $H_1,\dots,H_{n-1}$ be Cartier divisors containing $z$ such that $(Z',H_1+\dots+H_{n_1};z)$ is simple normal crossing. 
Let $H_{X,i}$ be the pull-back of $H_i$ to $X$.
Then, the pair $(X'',S_0''+S_\infty''+\sum_{i=1}^{n-1} H_{X,i})$ is log canonical and log Calabi--Yau over $Z'$.
By~\cite[Theorem 1]{MS21}, we conclude that $\phi''\colon X''\rightarrow Z'$ is formally toric over a neighborhood of $z\in Z'$.
Further, all the reduced fibers of the toric morphisms are isomorphic to $\pp^1$.
Thus, the morphism $X''\rightarrow Z'$ is formally isomorphic to 
$X(\Sigma)\rightarrow Z(\sigma)$
with $\sigma=\langle e_1,\dots,e_{n-1}\rangle$.
The fan $\Sigma$ has maximal cones 
$\langle (e_1,p_1),\dots,(e_{n-1},p_n),(0,1)\rangle$
and $\langle (e_1,p_1),\dots,(e_{n-1},p_{n-1}),(0,-1)\rangle$
where the $p_i$ are integers.
The fact that each $p_i$ is an integer follows from the fact that $\phi''$ has no multiple fibers over codimension one points. 
We conclude that $X(\Sigma)$ is smooth, so $X''$ is smooth on a neighborhood of the fiber over $z\in Z'$. Since $z\in Z''$ is any closed point, we conclude that $X''$ is smooth. Therefore, $S_0''$ is a Cartier divisor that restricts to a single point on general fibers. By~\cite[Lemma 2.12]{Fuj87}, we conclude that $X''$ is a projectivized rank $2$ vector bundle over $Z'$.
Since $\phi''$ admits a section $S''_\infty$, which is disjoint from $S''_0$, we conclude that such vector bundle splits.\\

\textit{Step 4:} In this step, we show that the pair $(X,S_0+S_\infty)$ admits a $\mathbb{G}_m$-action.\\

Since $X''\rightarrow Z'$ is the projectivization of a rank $2$ split vector bundle, we have $\mathbb{G}_m\leqslant {\rm Aut}_{Z'}(X'')$.
Since $S''_0$ and $S''_\infty$ are sections of the vector bundle, we get 
$\mathbb{G}_m\leqslant {\rm Aut}_{Z'}(X'',S''_0+S''_\infty)$.
In particular, we have
$\mathbb{G}_m\leqslant {\rm Aut}_{Z}(X'',S''_0+S''_\infty)$.
Let $(X'',B'')$ be the log pull-back of $(X,B)$ to $X''$.
Then, we get 
$\mathbb{G}_m\leqslant {\rm Aut}_Z(X'',B'')$.
Indeed, $B''$ consists of $S_0''+S_\infty''$ and divisors that are vertical over $Z'$.
The birational map
$X\dashrightarrow X''$
only extract divisors with log discrepancy at most one with respect to $(X'',B'')$
whose centers are contained in either $S''_0$ or $S''_\infty$.
Note that both $S''_0$ and $S''_\infty$ are contained in the fixed point locus of $\mathbb{G}_m$. 
Thus, we can extract the exceptional divisors of $X\dashrightarrow X''$ torus equivariantly from $X''$ (see Lemma~\ref{lem:canonical-places}).
More precisely, there is a $\mathbb{G}_m$-equivariant projective birational crepant morphism $(Y,B_Y)\rightarrow (X'',B'')$
such that $Y\dashrightarrow X$ is a contraction over $Z$. By Lemma~\ref{lem:descending-torus-actions}, we conclude that $\mathbb{G}_m \leqslant {\rm Aut}_Z(X,B)$ and so 
$\mathbb{G}_m \leqslant {\rm Aut}_Z(X,S_0+S_\infty)$.\\

\textit{Step 5:} In this step, we finish the proof of the $\qq$-factorial case by using the torus action on $(X,S_0+S_\infty)$.\\

Note that all the reduced fibers of $\phi\colon (X,S_0+S_\infty)\rightarrow Z$ are isomorphic to $(\pp^1,\{0\}+\{\infty\})$.
Furthermore, we have 
$\mathbb{G}_m\leqslant {\rm Aut}_Z(X,S_0+S_\infty)$.
We conclude that the $\mathbb{G}_m$-orbits are contained in fibers of $\phi$.
In particular, one-dimensional orbit closures of the $\mathbb{G}_m$-action on $X$ are disjoint.
Let $Z_0$ be the normalized Chow quotient of $X$ by the $\mathbb{G}_m$-action.
Let $\mathcal{S}$ be a divisorial fan on $Z_0$
such that there is a $\mathbb{G}_m$-equivariant isomorphism $X(\mathcal{S})\simeq X$ (see~\cite[Theorem 5.6]{AHS08}).
Since orbit closures of $X$ are disjoint, every polyhedral divisor $\mathcal{D}\in \mathcal{S}$ has affine support. Let $\zz$ be the weighted monoid of the $\mathbb{G}_m$-action on $X$. Let $\mathcal{D}$ be a p-divisor on $\mathcal{S}$ with weighted cone $\qq_{\geq 0}$. Then, the variety $\widetilde{X}(\mathcal{D})\simeq X(\mathcal{D})$ is equivariantly isomorphic to ${\rm Spec}_{U_0}(\bigoplus_{m\geq 0}\mathcal{O}_X(mD_0))$ for some $\qq$-Cartier $\qq$-divisor $D_0$ on $U_0$ an affine open of $Z_0$ (see, e.g.,~\cite[Theorem 1]{Wat81}). By~\cite[Corollary 8.12]{AH06}, these $\mathbb{G}_m$-invariant open affine varieties equivariantly glue to ${\rm Spec}_{Z_0}(\bigoplus_{m\geq 0}\mathcal{O}_X(mD))$ for some $\qq$-Cartier $\qq$-divisor $D$ on $Z_0$. 
Analogously, the p-divisors of $\mathcal{S}$ with weighted cone $\qq_{\leq 0}$ equivariantly glue to 
${\rm Spec}_{Z_0}(\bigoplus_{m\leq 0} \mathcal{O}_X(mD))$. 
We conclude that $X$ is isomorphic to the equivariant gluing of ${\rm Spec}_{Z_0}(\bigoplus_{m\geq 0} \mathcal{O}_X(mD))$
and ${\rm Spec}_{Z_0}(\bigoplus_{m\leq 0} \mathcal{O}_X(mD))$ along
${\rm Spec}_{Z_0}(\bigoplus_{m\in \zz} \mathcal{O}_X(mD))$
Thus, $X$ is isomorphic to $\mathbb{P}(\mathcal{O}_{Z_0}\oplus \mathcal{O}_{Z_0}(D))$. 
Note that the fibers of $X\rightarrow Z_0$ are precisely the one-dimensional orbit closures for the torus action. 
This implies that $Z_0\simeq Z$. This finishes the proof in the $\qq$-factorial case.\\

\textit{Step 6:} In this step, we finish the proof in the general case.\\ 

Let $(X,B)\rightarrow Z$ be a standard $\pp^1$-link. 
Since $(X,B)$ is plt, it admits a small $\qq$-factorialization $Y\rightarrow X$.
Let $(Y,B_Y)$ be the log pull-back of $(X,B)$ to $Y$.
We run a $K_{Y}$-MMP over $Z$.
Let $Y\dashrightarrow X'$ be the steps of this MMP
and $X'\rightarrow Z'$ be the induced Mori fiber space.
Note that every prime divisor on $Y$ that is vertical over $Z$ is the closure of the pull-back of a prime divisor on $Z$. 
Then, the relative MMP $Y\dashrightarrow X'$ over $Z$ contracts no divisors.
Let $B'$ be the push-forward of $B_Y$ to $X'$
so $(X',B')$ is a plt pair
and $\lfloor B'\rfloor$ has two disjoint sections over $Z'$.
By Lemma~\ref{lem:MFS-two-disjoint-sect-is-p1}, we conclude that $(X',B')\rightarrow Z'$ is a standard $\pp^1$-link with $X'$ being $\qq$-factorial.
By the previous step, we have 
$\mathbb{G}_m\leqslant {\rm Aut}_{Z'}(X',B')$.
In particular, we have 
$\mathbb{G}_m\leqslant {\rm Aut}_{Z}(X',B')$.
By Lemma~\ref{lem:descending-torus-actions}, we get
$\mathbb{G}_m\leqslant {\rm Aut}_Z(X,B)$.
Then, the same argument as in Step 5 finishes the proof.
\end{proof} 

Now, we turn to prove a structural theorem for standard $\pp^1$-links for low-dimensional varieties.

\begin{proof}[Proof of Theorem~\ref{introthm:bir-model-p1-link}]
Let $X$ be a variety of dimension at most three.
Let $(X,B)\rightarrow Z$ be a standard $\pp^1$-link.
Let $(Z,B_Z)$ be the log pair induced on $Z$ by the canonical bundle formula.
Note that $Z$ is either a surface or a curve.
By Theorem~\ref{introthm:p1-link-char}, we know that $X$ is the projectivization of a split $\qq$-vector bundle of rank $2$ on $Z$. 
Thus $X$ is isomorphic to 
$\mathbb{P}(\mathcal{O}_Z\oplus \mathcal{O}_Z(D))$ for a $\qq$-Cartier $\qq$-divisor $D$ on $Z$.
Let $\pi\colon Z'\rightarrow Z$ be a minimal resolution of $Z$. Let $(Z',B_{Z'})$ be the log pull-back of $(Z,B_Z)$ to $Z'$.
Let $D'$ be the pull-back of $D$ to $Z'$
and $X'\simeq \pp(\mathcal{O}_{Z'}\oplus \mathcal{O}_{Z'}(D'))$.
Let $\phi'\colon X'\rightarrow Z'$ be the induced fibration. Let $B':={\phi'}^*B_Z+S'_0+S'_\infty$,
where $S'_0$ and $S'_\infty$ are the strict transforms
of $S_0$ and $S_\infty$ in $X'$, respectively.
Then, we have a commutative diagram:
\[
\xymatrix{
(X,B)\ar[d]_-{\phi} & (X',B')
\ar[l]_-{\pi'}
\ar[d]^-{\phi'}\\ 
(Z,B_Z) & (Z',B_{Z'}) \ar[l]^-{\pi} 
}
\]
where the vertical arrows are crepant contractions
and the horizontal arrows are crepant birational morphisms.
Proceeding as in Step 2 of the proof of Theorem~\ref{introthm:p1-link-char}, we get a commutative diagram:
\[
\xymatrix{
(X,B)\ar[d]_-{\phi} & (X',B') 
\ar[l]_-{\pi'}
\ar[d]^-{\phi'} & (X'',B'')\ar[dl]^-{\phi''}\ar[l]_-{\pi''} \\ 
(Z,B_Z) & (Z',B_{Z'}) \ar[l]^-{\pi} & 
}
\]
where 
and $X''\rightarrow Z'$ is the projectivization of a rank $2$ split vector bundle.
In particular, we have that $\lfloor B''\rfloor=S''_0+S''_\infty$ are two disjoint smooth sections of $X''\rightarrow Z'$.
Let $k:=(S''_0)^2$. 
If $k=0$, then $X''$ is a trivial projective bundle and we are done.
Assume that $k>0$. 
We blow-up $X''$ at $k$ smooth points along $S''_0$
and blow down the strict transform of the corresponding fiber over $Z'$.
By doing so, we obtain a crepant birational map
$\pi'''\colon (Z',B_{Z'})\times (\pp^1,\Sigma^1)\dashrightarrow (X'',B'')$ over 
$Z'$. The product $(Z',B_{Z'})\times (\pp^1,\Sigma^1)$ carries the structures of a trivial $\pp^1$-bundle over $Z'$ and so we obtain a commutative diagram:
\[
\xymatrix{
(X,B)\ar[d]_-{\phi} & (X',B') 
\ar[l]_-{\pi'}
\ar[d]^-{\phi'} & (X'',B'')\ar[dl]_-{\phi''}\ar[l]_-{\pi''} & (Z',B_{Z'})\times (\pp^1,\Sigma^1) \ar[l]_-{\pi'''} \ar[lld]^-{\pi_1} \\ 
(Z,B_Z) & (Z',B_{Z'}) \ar[l]^-{\pi}. & &
}
\]
Note that the composition $\psi^{-1}:= \pi'''\circ \pi''\circ \pi'\colon (Z',B_{Z'})\times (\pp^1,\Sigma^1)\dashrightarrow (X,B)$ is a crepant birational map. This finishes the proof.
\end{proof}

\section{Birational complexity of log Calabi--Yau 3-folds}

In this section, we prove that a log Calabi--Yau $3$-fold of index one and coregularity zero
has birational complexity different than $1$.

\begin{proof}[Proof of Theorem~\ref{introthm:values-cbir-3-folds}]
Let $(X,B)$ be a log Calabi--Yau $3$-fold of index one and coregularity zero.
Since $(X,B)$ has index one its birational complexity is an integer.
By~\cite[Theorem 1.4]{MM24}, we know that $c_{\rm bir}(X,B)\in \{0,1,2,3\}$.
Thus, it suffices to show that the birational complexity of $(X,B)$ is not one. 
Assume, for the sake of contradiction, that
$c_{\rm bir}(X,B)=1$.
By~\cite[Theorem 1.1]{Mor24a}, there exists a crepant birational model $(X',B')$ of $(X,B)$
and a sequence of Mori fiber spaces 
\[
\xymatrix{ 
X' \ar[r]^-{\pi_1} & Z_1 \ar[r]^-{\pi_2} & Z_2 \ar[r]^-{\pi_3} & {\rm Spec}(\kk).
}
\]
where at least two of the three fibrations
$\pi_1,\pi_2$ and $\pi_3$ are strict conic fibrations.
Let $(Z_1,B_{Z_1})$
and $(Z_2,B_{Z_2})$ be the log Calabi--Yau 
pairs induced by the canonical bundle formula.
We may assume that $c_{\rm bir}(X,B)=c(X',B')$.

First, assume that $\pi_1$ is a strict conic fibration.
By~\cite[Lemma 2.36]{Mor24a}, we may assume that $(Z_1,B_{Z_1})$ is $\qq$-factorial and dlt.
Hence, by Lemma~\ref{lem:strict-conic-bundle-is-p1-link}, we conclude that $(X',\Gamma') \rightarrow Z_1$ is a standard $\pp^1$-link for a suitable effective divisor $0\leq \Gamma'\leq B'$. By Theorem~\ref{introthm:p1-link-char}, we conclude that $X'\rightarrow Z_1$ is a projectivized split $\qq$-vector bundle.
In particular, the log Calabi--Yau pair $(Z_1,B_{Z_1})$ is a log Calabi--Yau pair of coregularity zero and index one.
Hence, we have $(Z_1,B_{Z_1})\simeq_{\rm cbir} (\pp^2,\Sigma^2)$.
Note that the birational map $\phi\colon \pp^2 \dashrightarrow Z_1$ only extracts log canonical places
and canonical places.
Each such canonical place (resp. log canonical place)
induce a canonical place (resp. log canonical place)
on $(X',B')$ by adjunction.
Thus, there is a crepant birational map
$\phi_X \colon (X'',B'')\dashrightarrow (X',B')$,
with $B''\geq 0$, 
making the following diagram commutative:
\[
\xymatrix{
(X',B')\ar[d]_-{\pi_1} & (X'',B'')\ar[d]^-{\pi_1'}\ar@{-->}[l]_-{\phi_X} \\
(Z_1,B_{Z_1}) & (\pp^2,\Sigma^2)\ar@{-->}[l]^-{\phi}
}
\]
where $\pi'_1\colon (X'',B'')\rightarrow \pp^2$ is a strict conic fibration.
In particular, we get $\rho(X'')=2$
and $|B''|=5$.
We conclude that $(X'',B'')$ is a toric pair
and so $c_{\rm bir}(X,B)=0$,
leading to a contradiction.

From the previous paragraph, we know that $\pi_1$ is not a strict conic fibration.
Hence, $B'$ contains a unique prime divisor that dominates $Z_1$.
In this case, we have $c(X',B')=2$ 
and $\rho(X')=\rho(Z_1)+1$.
We conclude that $|B'|= \rho(Z_1)+3$.
Thus $|B'|$ has at least $\rho(Z_1)+2$ prime components that are vertical over $Z_1$.
Since $X'\rightarrow Z_1$ is a Mori fiber space, the divisor $\lfloor B_{Z_1}\rfloor$ has at least $\rho(Z_1)+2$ prime components.
We conclude that $\lfloor B_{Z_1}\rfloor$ has at least $\rho(Z_1)+2$ prime components.
Therefore, $(Z_1,B_{Z_1})$ is a toric log Calabi--Yau pair.
In particular, there is a crepant birational map
$(\pp^2,\Sigma^2)\dashrightarrow (Z_1,B_{Z_1})$ that only extract log canonical places.
By~\cite[Lemma 2.1]{MM24}, there is a crepant birational model $(X'',B'')$ of $(X',B')$ that admits a crepant fibration to $(\pp^2,\Sigma^2)$.
Let $\pi'_1\colon (X'',B'')\rightarrow  (\pp^2,\Sigma^2)$ 
be the corresponding fibration.
We have that ${\pi'_1}^*(\Sigma^2)$ is contained in the support of $B''$.
In particular, $B''$ has four prime components;
three prime components that are vertical over $\pp^2$
and dominate the prime components of $\Sigma^2$,
and a horizontal prime component that maps $2$-to-$1$ to $\pp^2$.
Let $S$ be the horizontal component of $B''$
and $S_0,S_1$, and $S_2$ be the vertical components.
Note that the pair $(X'',S)$ is plt.
Indeed, it has no log canonical centers, different from $S$,
that are horizontal over $\pp^2$
and it has not log canonical centers that are vertical over $\pp^2$. In particular, the divisor $S$ is normal (see~\cite[Proposition 5.51]{KM98}).
Let $(S,B_S)$ be the log Calabi--Yau pair obtained
by adjunction of $(X'',B'')$ to $S$.
Then the morphism $(S,B_S)\rightarrow (\pp^2,\Sigma^2)$ is a finite crepant morphism.
We conclude that $(S,B_S)$ is a toric log Calabi--Yau pair.

Since $(S,B_S)$ is a toric log Calabi--Yau surface,
each intersection $S\cap S_i$ is irreducible for $i\in \{0,1,2\}$. 
We argue that $S\cap S_i\rightarrow H_i$ is a $2$-to-$1$ cover for each $i\in \{0,1,2\}$.
Here, $H_i$ is the $i$-th coordinate hyperplane of $\pp^2$. It suffices to show the statement for $i=0$.
Let $(S_0,B_{S_0})$ be the slc Calabi--Yau pair induced by adjunction of $(X'',B'')$ to $S_0$ (see~\cite[Example 2.6]{FG12}).
Consider the fibration $f_0\colon S_0\rightarrow H_0$.
Since $S_0$ is prime, a general fiber $F$ of $f_0$ is an irreducible slc curve.
Let $(F,B_F)$ be the restriction of $(S_0,B_{S_0})$ to $F$.
Note that $(F,B_F)$ is an slc log Calabi--Yau pair.
Since $S\cap S_0$ intersects $F$, we conclude that $F$ is smooth. Otherwise, its normalization is not a curve of log Calabi--Yau type.
There are two possibilities; either $S\cap S_0\rightarrow H_0$ is $2$-to-$1$ or 
there is a $1$-dimensional log canonical center $Q$ of $(X'',B'')$,
that dominates $H_0$, contained in $S_0$ but not contained in $S,S_1$, or $S_2$.
In the latter case, the pair $(X'',S_0;Q)$ is not plt.
By the classification of $2$-dimensional log canonical singularities, we conclude that
$2(K_{X''}+B'')\sim 0$ and $K_{X''}+B''$ is not linearly trivial around $Q$ (see~\cite[\S 3, Figure 3]{Kol92}).
This contradicts the fact that $(X'',B'')$ has index one. We conclude that $S\cap S_i\rightarrow H_0$ is $2$-to-$1$ for each $i$.

Let $S_0^\vee \rightarrow S_0$ be its normalization
and $(S_0^\vee,B_0^\vee)$ be the log pull-back of $(S_0,B_{S_0})$ to $S_0^\vee$.
Note that $(S_0^\vee,B_0^\vee)$ is a log Calabi--Yau surface of index one and coregularity zero.
Let $S_0^\vee \rightarrow \pp^1$ be the contraction associated
to the Stein factorization of $S_0^\vee\rightarrow H_0$.
Then, the log Calabi--Yau pair induced on $\pp^1$ by the canonical bundle formula is $(\pp^1,\Sigma^1)$.
Up to a crepant birational transformation over $\pp^1$, 
we may assume that $S_0^\vee$ is $\qq$-factorial and $B_0^\vee$ has three prime components;
two vertical components $P_1$ and $P_2$ that map to $\{0\}$ and $\{\infty\}$ respectively, 
and one horizontal component $P$ that maps $2$-to-$1$ to $\pp^1$.
In particular, $P$ and $P_1$ intersect at a single point.
This implies that $P_1$ contains a closed point $q$ that is a log canonical center of $(S_0^\vee,B_0^\vee)$ which is not contained in either $P_2$ or $P$. 
Hence, the pair $(S_0^\vee,B_0^\vee)$ has a $D$-type singularity at $q$ (see~\cite[\S 3, Figure 3]{Kol92}).
This implies that $K_{S_0}^\vee + B_0^\vee$ does not have index one at $q$. 
Then, we get a contradiction.
We conclude that the birational complexity of $(X,B)$ is not one.
\end{proof} 

\section{Crepant fibrations to projective spaces}

In this section, we construct crepant fibration
to projective spaces for log Calabi--Yau $3$-folds
of index one and coregularity zero.

\begin{proof}[Proof of Theorem~\ref{introthm:crepant-model-3-fold}]
Let $(X,B)$ be a log Calabi--Yau $3$-fold 
of index one and coregularity zero.
By Theorem~\ref{introthm:values-cbir-3-folds}, we know that 
$c_{\rm bir}(X,B)\in \{0,2,3\}$.
If $c_{\rm bir}(X,B)=3$, then the statement is vacuous as we can simply consider the structure morphism $X\rightarrow {\rm Spec}(\kk)$.
If $c_{\rm bir}(X,B)=0$, then we know that $(X,B)\simeq_{\rm cbir}(\pp^3,\Sigma^3)$ by~\cite[Theorem 1.6]{MM24}.
Thus, it suffices to consider the case 
$c_{\rm bir}(X,B)=2$ and show that $(X,B)$ has a crepant
birational model that admits a crepant contraction to $(\pp^1,\Sigma^1)$.

Let $(X,B)$ be a log Calabi--Yau $3$-fold 
of index one and coregularity zero.
Assume that $c_{\rm bir}(X,B)=2$.
By~\cite[Theorem 1.1]{Mor24a}, there exists a crepant birational model $(X',B')$ of $(X,B)$ and a sequence of Mori fiber spaces
\[
\xymatrix{ 
X' \ar[r]^-{\pi_1} & Z_1 \ar[r]^-{\pi_2}  & \dots \ar[r]^-{\pi_k} & {\rm Spec}(\kk),  
}
\]
where $2\leq k\leq 3$ and at least one of the Mori fiber spaces
is a strict conic fibration.
If $\pi_k$ is a strict conic fibration, then the statement follows.
If $\pi_1$ is a strict conic fibration, then proceeding as in the proof of Theorem~\ref{introthm:values-cbir-3-folds}, we get $c_{\rm bir}(X,B)=0$, leading to a contradiction.
We conclude that $k=3$ and $\pi_2$ is a strict conic fibration.

From now on, we assume that $k=3$ and $\pi_2$ is a strict conic fibration.
Let $(Z_1,B_{Z_1})$ be the log Calabi--Yau pair induced by the canonical bundle formula.
Let $\pi_2\colon (Z_1,B_{Z_1})\rightarrow \pp^1$ be the strict conic fibration.
By~\cite[Corollary 3]{FMM22}, we know that $(Z_1,B_{Z_1})$ has index two and coregularity zero.
Furthermore, we may assume that $B'$ has a unique prime component $S$ that dominates $Z_1$ and $\pi_1^*\lfloor B_{Z_1} \rfloor$ is contained in the support of $B'$.
Note that $S$ is a normal divisor as
$(X',S)$ is plt.
Let $(S,B_S)$ be the log Calabi--Yau pair obtained by performing adjunction of $(X',B')$ to $S$.
Then, we have a crepant finite surjective morphism
$(S,B_S)\rightarrow (Z_1,B_{Z_1})$ of degree $2$.
Let $Z'\rightarrow Z_1$ be a minimal resolution.
Since $(Z_1,B_{Z_1})\rightarrow \pp^1$ is a strict conic fibration,
all the singularities of $Z_1$ are contained in $\lfloor B_{Z_1}\rfloor$.
In particular $Z'\rightarrow Z_1$ only extracts divisors with log discrepancy $\leq \frac{1}{2}$ with respect to $(Z_1,B_{Z_1})$.
Let $(Z',B_{Z'})$ be the log pull-back of $(Z_1,B_{Z_1})$ to $Z'$.
By Lemma~\ref{lem:2-to-1}, we have a crepant birational model
$(X'',B'')$ of $(X',B')$, with $B''$ effective, 
that admits a crepant fibration to $(Z_2,B_{Z_2})$, where $(Z_2,B_{Z_2})\rightarrow (Z_1,B_{Z_1})$ is a dlt modification.
Further, the morphism $X''\rightarrow Z_2$ is a Mori fiber space.
Note that $Z_2$ is a smooth surface.
Every fiber of $Z_2\rightarrow \pp^1$ that contains a log canonical center 
is contained in $\lfloor B_{Z_2}\rfloor$.
This follows from the connectedness of log canonical centers (see~\cite[Theorem 1.1]{FS20}).
Thus, when we run a $K_{Z_2}$-MMP over $\pp^1$, every curve contracted by the MMP
is either disjoint from the log canonical centers or
intersects positively a curve with coefficient one in $\lfloor B_{Z_2}\rfloor$.
Thus, by Lemma~\ref{lem:going-down}
and Lemma~\ref{lem:going-down-2}, we can find a crepant birational model $(X^{(3)},B^{(3)})$ of $(X'',B'')$, with $B^{(3)}$ effective, that admits a crepant contraction to $(Z_3,B_{Z_3})$.
The morphism $X^{(3)}\rightarrow Z_3$ is a Mori fiber space and $Z_3$ is a Hirzebruch surface.

First, we assume that every vertical component of $B_{Z_3}$ has coefficient $\frac{1}{2}$.
In this case, the pair $(Z_3,B_{Z_3})$ is plt.
We argue that there exists a crepant birational map
$(\pp^1\times \pp^1,\Gamma)\dashrightarrow (Z_3,B_{Z_3})$ that only extract divisors with log discrepancy $\leq \frac{1}{2}$ with respect to $(Z_3,B_{Z_3})$ and $\lfloor \Gamma\rfloor$ contains two disjoint sections for the projection to the first component.
Indeed, the log Calabi--Yau pair $(Z_3,B_{Z_3})$ has a prime component $P$ that is vertical over $\pp^1$ and has coefficient $\frac{1}{2}$ in $B_{Z_3}$.
The crepant birational map 
$(\pp^1\times\pp^1,\Gamma)\dashrightarrow (Z_2,B_{Z_2})$ can be obtained by inductively blowing up the intersection of $P$ 
and one of the sections of the Hizebruch surface, and then blowing down the strict transform of the corresponding fiber.
By Lemma~\ref{lem:2-to-1} and Lemma~\ref{lem:going-down-2}, there exists a crepant model $(X^{(4)},B^{(4)})$ of $(X^{(3)},B^{(3)})$, with $B^{(4)}$ effective, that admits a crepant fibration to $(\pp^1\times\pp^1,\Gamma)$.
Note that $(\pp^1\times \pp^1,\Gamma)$ admits a crepant fibration to $(\pp^1,\Sigma^1)$
by projecting to the second component.
Thus, the log Calabi--Yau pair
$(X^{(4)},B^{(4)})$, which is crepant birational equivalent to
$(X,B)$, admits a crepant fibration to $(\pp^1,\Sigma^1)$.

Finally, assume that there is a vertical component of $B_{Z_3}$ with coefficient $1$.
In this case, there exists a crepant birational map $(\pp^1\times \pp^1,\Gamma)\dashrightarrow (Z_3,B_{Z_3})$ that only extract log canonical places
and $\lfloor \Gamma\rfloor$ contains two disjoint sections for the projection to the first component.
Then, by~\cite[Lemma 2.10]{MM24} and 
Lemma~\ref{lem:going-down}, there exists a crepant model $(X^{(4)},B^{(4)})$ of $(X^{(3)},B^{(3)})$, with $B^{(4)}$ effective, that admits a crepant fibration to $(\pp^1\times\pp^1,\Gamma)$.
The pair $(\pp^1\times\pp^1,\Gamma)$ admits a crepant fibration to $(\pp^1,\Sigma^1)$ by projecting to the second component.
The log Calabi--Yau pair $(X^{(4)},B^{(4)})$ is crepant birational equivalent to $(X,B)$
and admits a crepant fibration to $(\pp^1,\Sigma^1)$.
\end{proof} 

\section{Examples and questions}

In this section, we collect some examples and propose some questions.
First, we give examples of log Calabi--Yau pairs $(X,B)$
of dimension $3$, index $1$, and coregularity $0$
with different birational complexities.

\begin{example}\label{ex:cbir=0}
{\em 
Let $(T,B_T)$ be any toric log Calabi--Yau pair, i.e., 
$B_T$ is the reduced complement of the algebraic torus on $T$.
Then, we have $c(T,B_T)=\dim T +\rho(T)-|B_T|=0$.
Hence, $(T,B_T)$ is a log Calabi--Yau pair of index $1$, coregularity $0$, and birational complexity $0$.
}
\end{example}

\begin{example}\label{ex:cbir=2}
{\em 
In this example, we construct a log Calabi--Yau pair $(X,B)$ of dimension $3$, index $1$, coregularity $0$, 
and birational complexity $2$.

Let $X$ be a smooth cubic hypersurface in $\mathbb{P}^4$.
Then, $X$ is not a rational variety~\cite[Theorem 0.12]{CG72}.
Let $H$ be a general hyperplane in $\mathbb{P}^3$.
The pair $(\pp^4,X+H)$ is a dlt Fano pair
of coregularity $2$.
Let $Z$ be a minimal log canonical center of $(\pp^4,X+H)$.
Then, $Z$ is a cubic surface in $\pp^3$.
In particular, there exists a $1$-complement of coregularity zero $(Z,C)$.
By Kawamata-Viehweg vanishing theorem, the $1$-complement $C$ lifts to a $1$-complement 
$C_X$ of the log Fano pair $(X,H|_X)$.
Thus, we have that $(X,H|_X+C_X)$ is a log Calabi--Yau $3$-fold of index one and coregularity zero.
Note that 
\[
c(X,H|_C+C_X)= \dim X + \rho(X) - |H|_C+C_X| 
\leq 2.
\]
We argue that $c_{\rm bir}(X,H|_C+C_X)=2$.
Assume otherwise that $(X,H|_C+C_X)$ admits a log
Calabi--Yau crepant birational model $(Y,B_Y)$
for which $c(Y,B_Y)<2$.
Then, by Theorem~\ref{introthm:values-cbir-3-folds}, we would have $c(Y,B_Y)=0$.
Indeed, the complexity of this log Calabi--Yau pair
must be an integer and cannot be equal to one (see Theorem~\ref{introthm:values-cbir-3-folds}).
Thus, the log Calabi--Yau pair $(X,H|_C+C_X)$ would have a toric model, and hence $X$ would be rational.
Leading to a contradiction.
We conclude that $c_{\rm bir}(X,H|_C+C_X)=2$.
}
\end{example}

\begin{example}\label{ex:cbir=3}
{\em 
In this example, we construct a log Calabi--Yau pair $(X,B)$ of dimension $3$, index $1$, coregularity $0$, and birational complexity $3$.

Let $X$ be the hypersurface defined by 
\[
\{
[x_0:x_1:x_2:x_3:x_4]\mid
x_3(x_0^3+x_1^3)+x_2^4+x_0x_1x_2x_3+x_4(x_3^3+x_4^3)=0
\} \subset \mathbb{P}^4_{x_0,...,x_4}.
\]
By~\cite[Example 3.2]{Kal20}, we know that $X$ is a terminal birationally superrigid Fano $3$-fold.
Let $B:=X\cap \{x_4=0\}$.
Then, the log Calabi--Yau pair $(X,B)$
has index one and coregularity zero.
As $X$ is birationally superrigid,~\cite[Corollary 1.2]{Mor24a} implies that $c_{\rm bir}(X,B)=3$.
}
\end{example}

We conclude by posing two natural questions about the values of the birational complexity
of higher-dimensional log Calabi--Yau pairs.

\begin{question}
Let $(X,B)$ be a log Calabi--Yau pair of dimension $n\geq 4$, index one, and coregularity zero.
What are the possible values for the birational complexity $c_{\rm bir}(X,B)$ of $(X,B)$? 
Are there such pairs with birational complexity equal to one?
\end{question} 

We know that every Fano variety $X$ of coregularity zero admits either a $1$-complement or a $2$-complement of coregularity zero (see, e.g.,~\cite[Theorem 4]{FFMP22}).
Thus, we can extend the previous question to this setting.

\begin{question}
Let $(X,B)$ be a log Calabi--Yau pair of dimension $n\geq 2$, index two, and coregularity zero. 
What are the possible values for the birational complexity $c_{\rm bir}(X,B)$ of $(X,B)$?
\end{question} 

In the previous question, it is unclear which fractional numbers occur and how this reflects on the geometry of the pair. For instance, in dimension one, the only possible birational complexity of pairs of coregularity zero is zero.

\bibliographystyle{habbvr}
\bibliography{references}

\end{document}